\documentclass[12pt]{amsart}

\usepackage{amsfonts}
\usepackage{amssymb}
\usepackage{graphicx}
\usepackage{amsmath,mathtools}
\usepackage{latexsym}
\usepackage{amscd}
\usepackage{xypic}
\usepackage{mathrsfs}
\usepackage{enumitem} 
\usepackage{braket}
\usepackage[margin=3cm]{geometry}
\usepackage{amsthm}
\usepackage{accents}
\usepackage{xcolor}

\usepackage{hyperref}

\setlist{itemsep=3pt}

\allowdisplaybreaks

\newtheorem{prop}{Proposition}
\newtheorem{theo}[prop]{Theorem}
\newtheorem*{theo*}{Theorem}
\newtheorem{lemm}[prop]{Lemma} 
\newtheorem*{lemm*}{Lemma}
\newtheorem{coro}[prop]{Corollary} 

\theoremstyle{definition}

\newtheorem{rema}[prop]{Remark}
\newtheorem{defi}[prop]{Definition}
\newtheorem*{defi*}{Definition}

\numberwithin{equation}{section}

\numberwithin{prop}{section}

\newcommand{\p}{\partial}

\newcommand{\BB}{\mathbb{B}}

\newcommand{\NN}{\mathbb{N}}

\newcommand{\RR}{\mathbb{R}} 

\newcommand{\cA}{\mathcal A}

\renewcommand{\cH}{\mathcal H}

\newcommand{\cS}{\mathcal S}

\DeclareMathOperator{\Ric}{Ric}

\DeclareMathOperator{\Vol}{Vol}
\DeclareMathOperator{\dive}{div}

\newcommand{\sff}{\mathrm{I\!I}}

\usepackage{graphicx}
\allowdisplaybreaks

\newenvironment{nouppercase}{%
  \renewcommand{\uppercasenonmath}[1]{}}{}

\title{Rigidity of Complete Free Boundary Minimal Hypersurfaces \\ in Convex NNSC Manifolds}

\author{Yujie Wu}
\address{Department of Mathematics, Bldg.\ 380, Stanford University, Stanford, CA 94305, USA}
\email{yujiewu@stanford.edu}

\begin{document}
\begin{nouppercase}
	\maketitle
\end{nouppercase}

\begin{abstract}
We prove that in the unit ball of $\RR^4$, there is no complete two-sided stable free boundary immersion. The result follows from a rigidity theorem of complete free boundary minimal hypersurfaces in complete 4-manifolds with non-negative intermediate Ricci curvature, convex boundary and weakly bounded geometry. The method uses warped $\theta$-bubble, a generalization of capillary surfaces.
\end{abstract}

\section{Introduction}
Recall a complete two-sided stable  minimal (free boundary) immersion $M^{n-1}\hookrightarrow X^n$ satisfies the following stability inequality, for any compactly supported test function $\phi$,
\begin{equation}\label{stabIneq}
	\int_M |\nabla \phi|^2 -(\Ric_X(\nu_M,\nu_M)+|\sff_M|^2)\phi^2-\int_{\p M} \sff_{\p X}(\nu_M,\nu_M)\phi^2\geq 0,
\end{equation}
for a choice of unit normal $\nu_M$ of $M\hookrightarrow X$, $\Ric_X$ the Ricci curvature of $X$ and $\sff$ the corresponding second fundamental forms. The positivity of the curvature tensor of the ambient manifold $X^n$ has been exploited to obtain rigidity or non-existence results of   stable minimal (free boundary) hypersurfaces.
When $X^n$ has non-negative Ricci and M is closed, then $M$ must be totally geodesic (\cite{Simons1968MinimalVar}); when $X^n$ has positive scalar curvature (PSC) and M is closed, then  $M$ must also have PSC (\cite{SchoenYau1979Incompressible},\cite{schoen1979structure}). See \cite{chodosh2024complete} for a more complete review of the literature.

In the case that $M$ or $X$ is non-compact, results have been obtained for surfaces in a three manifold. 
If the scalar curvature has $R_X \geq 0$ and $\p X =\emptyset$, then $M$ with the induced metric is either conformal to a plane or a cylinder, and the later case implies that $M$ is totally geodesic, intrinsically flat, $R_X \rvert_M=0$ and $ \Ric_X(\nu_M,\nu_M)=0$ along $M$ (\cite{FischSchoe1980-StructureCompleteStable}). Furthermore, if $R_X \geq 1$, then $M$ must be compact (\cite{SchoenYauBCM},\cite{GroLawPSCD}), and also admit a metric of PSC. The idea that the positivity of scalar curvature can be ``inherited'' has seen many fruitful applications (\cite{SchYauPMT}, \cite{chodoliSoapBubb}).

In the case that $\p X \neq \emptyset$, Franz (Proposition 3.2.5 in \cite{franz2022contributions}) obtained results for manifolds with ``weakly positive geometry''. Denote $H_{\p X}$ to be the mean curvature. If we assume $R_X \geq R_0> 0$, $H_{\p X} \geq 0$ and $\p X$ has no minimal components, or assume that $R_X \geq 0$ and $H_{\p X}\geq H_0>0$, then every complete two-sided stable free boundary minimal surface $M$, must be compact with intrinsic $\text{diam}(M)\leq C(H_0,R_0)$, and is diffeomorphic to a disc. The result allows the author to obtain compactness results for compact embedded finite index minimal surfaces in a compact 3-manifold with weakly positive geometry, leading to a uniform bound of area, total curvature, genus and boundary components (Theorem 3.2.1 in \cite{franz2022contributions}).

In higher dimensions, to obtain analogous rigidity or non-existence results, the assumption of $\Ric >0$ or $R_g \geq 1 $ needs to be strengthened (see examples in \cite{chodosh2024complete}). The authors in \cite{chodosh2024complete} combined PSC and non-negative intermediate Ricci curvature to obtain rigidity results for stable minimal hypersurfaces in a 4-manifold. The free boundary case using the same technique of PSC is obtained in \cite{wu2023free}.
In this paper, our main result shows that instead of assuming positive scalar curvature of $X$, we can assume positive mean curvature of $\p X$.

\begin{theo}\label{rigidity}
	Consider a 4-manifold $(X^4,\p X)$ with weakly bounded geometry, assume  $\Ric^X_2\geq 0, \sff_{\p X}\geq 0$ and $H_{\p X}\geq H_0>0$.
	If $(M^3,\p M)\hookrightarrow (X^4,\p X)$ is a complete two-sided stable free boundary minimal immersion, then $M$ is totally geodesic and $\Ric_X(\nu_M,\nu_M)=0$ along $M$, $\sff_{\p X}(\nu_M,\nu_M)=0$ along $\p M$.
\end{theo}

We will elaborate the assumptions of weakly bounded geometry and the intermediate curvature condition $\Ric^X_2$ in the next section. In particular, a compact 4-manifold with non-negative sectional curvature and convex boundary, for example the unit ball in $\RR^4$, satisfies our assumptions.

\begin{coro}
	There is no complete two-sided stable free boundary minimal immersion in $\BB^4$.
\end{coro}

The method that enables the study of PSC manifolds in \cite{chodosh2024complete} is called $\mu$-bubbles, introduced in Gromov's 1996 paper \cite{gromov1996}. 
Recently, the $\mu$-bubble method has been revisited to obtain prolific results including the Generalized Geroch Conjecture, the Stable Bernstein Theorem conjectured by Schoen and the Multiplicity One Conjecture of Marques and Neves
(\cite{gromov2019scalarcurvaturemanifoldsboundaries}, \cite{gromovMetricInequalitiesScalar2018a}, \cite{zhu2021WidthEstimate}, \cite{zhu2023Rigidity},\cite{zhou2019MulOneConjec},\cite{chodoliSoapBubb}, \cite{chodoshli2024acta},\cite{chodoshli2023FPi}).

The method of $\mu$-bubble, as a generalization of minimal hypersurfaces, has proven to be a powerful tool to analyze the geometry of the ambient manifold.

In the case of manifolds with boundary, the analogy is given by capillary surfaces, which means constant mean curvature (CMC) surfaces  having constant contact angle with the ambient manifolds' boundary.
Li (\cite{li2020polyhedron}) used the method of capillary surfaces to show Gromov's dihedral rigidity conjecture for conical and prism type polyhedron. Chai and Wang (\cite{chaiwangDihedral}) confirmed the conjecture for some cases of hyperbolic 3-manifold.
 
In \cite{wu2024capillary}, the author used capillary surfaces with prescribed (varying) contact angle, a notion called ``$\theta$-bubble'', to study comparison results and geometry of manifolds with non-negative scalar curvature (NNSC) and uniformly mean convex boundary, obtaining sharp comparison results for surfaces, a 1-Urysohn width bound, a decomposition for the boundary of such manifolds, and a bandwidth estimate.
Using capillary surfaces with prescribed contact angle, Ko and Yao (\cite{koyao2024scalar}) proved a smooth comparison and rigidity result analogous to Gromov's dihedral rigidity conjecture. 

The method of $\theta$-bubble in the above work for surfaces in 3-manifolds can be used inductively. We observe that in manifolds with NNSC and uniformly mean convex boundary, we have the PSC equivalent ``inheritance''phenomenon. For a precise (and more general) statement, see Lemma \ref{inherit}.

Since we are now interested in hypersurfaces in a non-compact ambient manifold, it's important to control the number of ends of its stable  minimal hypersurfaces. Cao, Shen and Zhu (\cite{cao1997structure}) proved that a complete stable minimal hypersurfaces in $\RR^4$ has at most one end; a similar result on non-parabolic ends is obtained for manifolds with non-negative intermediate Ricci curvature in \cite{chodosh2024complete}. The case with FBMH $(M^3,\p M) \hookrightarrow (X^4,\p X)$ requires additional control of the boundary. We would like to use $\theta$-bubbles (near the boundary) to exhaust the non-parabolic end.
But a priori we don't know if the boundary $\p M$ is connected or disconnected, compact or non-compact, whether it's contained in the parabolic ends or non-parabolic end.

Using  Theorem \ref{NcptBound}, we can show if $(M^3,\p M) \rightarrow (X^4,\p X)$ is a FBMH and $\Ric^X_2\geq 0, \sff^{\p X}_2\geq 0$, assume $M$ is non-parabolic and $\Vol(M)=\infty$, then each component of $\p M$ must be non-compact, and each component of the boundary $\p M$ must has an end in the only non-parabolic end of $M$.

We now provide a sketch of the proof of Theorem \ref{rigidity}. We may pass to the universal cover and assume $M$ is simply connected. The idea is similar to \cite{chodosh2024complete}. We want to control the volume growth of a stable FBMH $M^3$ in $X^4$ as of Theorem \ref{rigidity}. The parabolic ends behave  in a good way as one can find a harmonic function that approaches $1$ everywhere while the Dirichlet energy goes to $0$ when exhausting the parabolic end, which serves as a good test function to plug into the stability inequality (\ref{stabIneq}). On a non-parabolic end, existence of a positive barrier function prevents us from using the same idea, so we want to exhaust a non-parabolic end with chunks with bounded diameter and volume, which shows that each non-parabolic end grows linearly in volume. This is achieved using $\theta$-bubbles together with the control of the boundary $\p M$ and the non-parabolic end.

In section 2, we state some preliminaries and introduce the notations and set up of the paper. In section 3, we define a ``warped $\theta$-bubble'' and derive its first and second variations. In section 4, we derive the inductive process and describe the inheritance phenomenon of warped $\theta$-bubbles. In section 5, we study  non-parabolic ends of FBMH in manifolds with non-negative intermediate Ricci curvature and 2-convex boundary conditions. In section 6, we prove the main results combining all the ingredients.

\textbf{Acknowledgements.} The author wants to thank her advisor Otis Chodosh for many insightful discussions and continuous encouragements in writing of this paper.

\section{Preliminaries}

We first set up some notations and definitions in this section. Recall that an immersed submanifold $M\hookrightarrow X$ is minimal if its mean curvature vanishes everywhere. Throughout the paper we use the convention $\sff_M(Y,Z)=-\langle \overline{\nabla}_ZY, \nu_M \rangle$ given a choice of unit normal vector $\nu_M$ for a two-sided hypersurface $M$ and $\overline{\nabla}$ the Levi-Civita connection on the ambient Riemannian manifold $X$. We define mean curvature as $H_M=\text{tr}(\sff_M)$. In this convention, the mean curvature of a sphere with outward pointing unit normal in the Euclidean ball is positive. We denote the induced connection on $M$ as $\nabla$.

We reserve the notation $\p M$ to denote the boundary of a manifold $(M,\p M)$, and if $\Omega$ is an open subset of $M$, we denote the topological boundary as $\p'\Omega$ and the closure of $\Omega$ as $\overline{\Omega}$. Then $\p' \Omega \cap \p M=\p'(\Omega \cap \p M)$ if $\Omega$ has Lipschitz boundary and $\overline{\Omega}$ intersect $\p M$ transversally. 

We first define the notion of free boundary minimal hypersurface (FBMH). Consider an immersion of hypersurface $(M,\p M) \hookrightarrow (X,\p X)$ and we always require $\p M \subset \p X$ and the immersion is two-sided.

\begin{defi}
	An immersion $(M,\p M) \hookrightarrow (X,\p X)$ is a FBMH if it's a critical point to the area functional among hypersurfaces with boundary in $\p X$. Equivalently, a FBMH has the following characterization,
	\begin{itemize}
		\item the mean curvature vanishes everywhere,
		\item the boundaries $\p M$ and $\p X$ intersect orthogonally, that is, the outward pointing unit normal of $\p M \subset M$ agrees with the outward pointing unit normal of $\p X \subset X$ (so the second fundamental form of $\p M \hookrightarrow M$ is the restriction of the second fundamental form of $\p X \hookrightarrow X$ on the tangent bundle $T \p M$).
	\end{itemize}
	The FBMH is stable if its second variation of area is non-negative, that is, for any compactly supported function $\varphi$ on $M$, we have the following stability inequality,
	\begin{equation}
		0\leq \int_M |\nabla \varphi|^2-(\Ric_X(\nu_M,\nu_M)+|\sff_M|^2)\varphi^2-\int_{\p M} \sff_{\p X}(\nu_M,\nu_M)\varphi^2.
	\end{equation}
\end{defi}

We denote the Ricci curvature of a Riemannian manifold $X$ as $\Ric_X$ and scalar curvature as $R_X$. We denote the Riemann curvature tensor as $R_X(\cdot,\cdot,\cdot,\cdot)$.

We now introduce the non-negative 2-intermediate Ricci curvature condition $\Ric^X_2\geq 0$, lying between non-negative sectional curvature and non-negative Ricci curvature. 

\begin{defi}
	We say a Riemannian manifold $X$ has $\Ric^X_2\geq 0$ if for any $x\in X$ and orthonormal vectors $u,v,w$ of $T_xX$,
	\begin{equation*}
		R_X(v,u,u,v)+R(w,u,u,w)\geq 0.
	\end{equation*}
\end{defi}
If the dimension of $X$ is at least 3, then $\Ric^X_2\geq 0$ implies $\Ric_X \geq 0$.

Combining $\Ric^X_2\geq 0$ and Gauss-Codazzi equation, a FBMH in $X$ has its Ricci bounded from below by its second fundamental form (see \cite{chodosh2024complete}, \cite{wu2023free}).
\begin{lemm}[\cite{chodosh2024complete}, Lemma 4.2]
	Let $(M^{n-1},\p M) \hookrightarrow (X^n,\p X)$ be a FMBH and $\Ric^X_2 \geq 0$, then
	\begin{equation*}
		\Ric_M \geq -\frac{n-2}{n-1} |\sff_M|^2.
	\end{equation*}
\end{lemm}

We can also define an analogous ``2-convexity'' condition for $\p X \hookrightarrow X$ (see \cite{wu2023free}).

\begin{defi}
	For a complete Riemannian manifold $(X,\p X)$, we say $\sff^{\p X}_2 \geq 0$ (or $\p X$ is 2-convex) if for any $x\in \p X$ and  orthonormal vectors $e_1,e_2$ of $T_x \p X$,
	\begin{equation*}
		\sff_{\p X}(e_1,e_1)+\sff_{\p X}(e_2,e_2)\geq 0.
	\end{equation*}
\end{defi}

We define the following ``weakly bounded geometry'' condition that holds for any compact manifolds or manifolds with bounded geometry (uniformly bounded curvature and injectivity radius).

\begin{defi}
	We say a complete Riemannian manifold $(X,\p X)$ has weakly bounded geometry (up to the boundary) at scale $Q>0$, if there is $\alpha \in (0,1)$ such that for any point $x\in X$,
	\begin{itemize}
		\item there is a pointed $C^{2,\alpha}$ local diffeomorphism $\Phi: (B_{Q^{-1}}(a),a) \cap \RR^n_+ \rightarrow (U,x) \subset X$ for some $a\in \RR^n_+:=\{(a_1,...,a_n)\in \RR^n, a_n \geq 0\}$;
		\item if $\p X \cap U \neq \emptyset $ then $ \Phi^{-1}(\p X \cap U)\subset \p \RR^n_+$;
	\end{itemize}
	furthermore, the map $\Phi$ satisfies,
	\begin{itemize}
		\item $e^{-2Q}g_0\leq \Phi^{*}g\leq e^{2Q}g_0$ as two forms, with $g_0$ the standard Euclidean metric;
		\item $\|\p_k \Phi^{*}g_{ij}\|_{C^{\alpha}}\leq Q$ where $i,j,k$ stand for indices in Euclidean space.
	\end{itemize}
\end{defi}

The following two Lemmas are useful consequences of the weakly bounded geometry condition. First is curvature estimates for stable FBMH following the stable Bernstein result of Chodosh and Li in dimension 4.

\begin{lemm}[\cite{wu2023free}, Lemma 3.1]\label{sffBound}
	Consider $(X^4,\p X,g)$ a complete Riemannian manifold with weakly bounded geometry, and $(M^3,\p M)\hookrightarrow (X^4,\p X)$ a FBMH, then there is a constant $C=C(X,g, Q)$ independent of $M$ such that, 
	\begin{equation*}
		\sup_{q\in M}|\sff_M(q)| \leq C 
	\end{equation*}
\end{lemm}

The second lemma is a volume control of balled of fixed radius.
\begin{lemm}[\cite{wu2023free}, Lemma 3.3]\label{volWBG}
	Let $(X^n,\p X)$ be a complete Riemannian manifold with weakly bounded geometry at scale $Q>0$, and $(M^{n-1},\p X)\hookrightarrow(X,\p X)$ a complete immersed submanifold with bounded second fundamental form, then for any $R>0$, there is a constant $C=C(R,Q)$ such that the volume of balls of radius $R$ around any point in $M$ is bounded by $C$, i.e. $\Vol(B^M_R(q))\leq C$ for any $q\in M$.
\end{lemm}

\section{Definitions of $\theta$-bubbles}

In this section we introduce the notion of $\theta$-bubble, a tool that is useful to probe the geometry of manifolds with non-negative scalar curvature (NNSC) and mean convex boundary, as an analogous notion to $\mu$-bubble, first utilized by Gromov to probe the geometry of manifolds with positive scalar curvature (PSC). 

We denote the reduced boundary of a Caccioppoli set $\Omega$ as $\p^{*}\Omega$. Then if $\Omega$ is an open set with smooth boundary of a Riemannian manifold, $\p^{*}\Omega=\p'\Omega$.

\begin{defi}
	Consider a Caccioppoli set $\Omega$ of a Riemannian manifold with boundary $(X^n,\p X)$, a ``$\theta$-bubble'' is a minimizer of the following ``prescribed contact angle'' problem, among variations that send $\p X$ to itself, 
	\begin{equation*}
		\cA_{\theta}(\Omega)=\cH^{n-1}(\p^{*} \Omega)-\int_{\p X \cap \Omega} \cos \theta,
	\end{equation*} 
	given a smooth function $\theta:\p X \rightarrow \RR$.
\end{defi}

Capillary surfaces are surfaces of constant mean curvature and constant contact angle, and they are a critical point of $\cA_{\theta}$ when $\theta$ is a constant function on $\p X$ and the volume ratio separated by the capillary surfaces is fixed. The $\theta$-bubbles can be thought as a ``generalized capillary surfaces''.

A more generalized notion called ``warped $\theta$-bubble'' is adapted from the definition of $\theta$-bubble with a weight on the ambient manifold, as an analogy to ``warped $\mu$-bubble''.

\begin{defi}
	Consider a Caccioppoli set $\Omega$ of a Riemannian manifold with boundary $(X^n,\p X)$, a `` warped $\theta$-bubble'' is a critical point to the following functional, among variations that send $\p X$ to itself, 
	\begin{equation*}
		\cA_{u}(\Omega)=\int_{\p^{*} \Omega} u d\cH^{n-1}-\int_{\p X \cap \Omega} u\cos \theta,
	\end{equation*} 
	given a smooth function $\theta:\p X \rightarrow \RR$ and a positive smooth function $u: X \rightarrow \RR_+$.
\end{defi}

We first show that a minimizer exists under suitable assumptions and is smooth up to codimension 4.

\begin{theo}\label{MPtheta}
	If $(N^n,\p N)$ is a compact connected Riemannian manifold with connected mean convex boundary $H_{\p N}>0$, let $\theta:\p N \rightarrow \RR$ be a smooth function such that $S_{\pm}=\{x \in \p N, \cos \theta(x)=\pm 1\}$ and $S_+$ and $S_-$ are non-empty closed domains in $\p N$ with smooth boundary. Let $\cS$ be the set of all Caccioppoli sets that contain $S_+$ and be disjoint from $S_-$, then
	\begin{equation*}
		\alpha=\inf_{\Omega \in \cS} \cA_{\theta}(\Omega)
	\end{equation*}
	is obtained by some $\Omega \in \cS$. If $n\leq 4$, then the minimizer $\Omega$ is a smooth submanifold with boundary that intersect $\p N$ transversally.
\end{theo}
\begin{proof}
	We follow the arguments in \cite{wu2024capillary} for surfaces and 3-manifolds. It's enough to show that there is a fixed open neighborhood $\Omega_+$ of $S_+$ such that any minimizing sequence must contain $\Omega_+$, and be disjoint from a fixed open neighborhood $\Omega_-$ of $S_-$. Then by connectedness of $\p N$, a minimizer must exists and must intersect $\p N$ in a compact subset of $\{x\in \p N, -1<\cos \theta(x)<1\}$ and the regularity results of \cite{chodosh2024improved} (see also \cite{philippis2015regularity}) applies, for $n\leq 4$ the minimizer is smooth and intersect $\p N$ transversally.
	
	Consider the following family of foliations $\Phi_t(x):=\exp(-\varphi_t(x)\nu_{\p N})$ where $\nu_{\p N}$ is the outward pointing unit normal. Here $\varphi_t(x):=\{t_0\phi(x)+t,0\}$ for some small fixed $t_0>0$ to be chosen, and $\phi:\p N\rightarrow[-1,1]$ is a smooth function such that $S_+=\{x\in \p N, \phi(x)\geq 0\}$ and $\nabla\phi(x)\neq 0$ for any $x\in \p S_+$. Note that $\Gamma_t:=\overline{\Phi_t(\p N) \setminus \p N}$ is a smooth submanifold in $N$, and as $t_0 \rightarrow 0, t\leq t_0$, $\Gamma_t$ converge to $S_+$ smoothly, so $H_{\p N}>0$ implies $H_t:=H_{\Gamma_t}>\epsilon$ for some $\epsilon>0$. 
	
	Denote $\Omega_t:=\cup_{-t_0\leq s\leq t} \Gamma_t$ with outward pointing unit normal $\nu_{\Gamma_t}$ along the boundary $\Gamma_t$. Then  we have that $\nu_{\p N}(x)\cdot \nu_{\Gamma_0}(x)> -1 =-\cos\theta(x)$ for any $x\in \p S_+$, which implies for small $t_0$ and $0\leq t\leq t_0$, $\nu_{\p N}(x)\cdot \nu_{\Gamma_t}(x)+\cos\theta(x) >0$ for any $x\in \p \Gamma_t$.
	
	Then we take any candidate $\Omega$ (without loss of generality assume $\p \Omega$ smooth) and compare,
	\begin{align*}
		\cA_{\theta}(\Omega_t\cup \Omega)-\cA_{\theta}(\Omega)&= \cH^{n-1}(\p^*(\Omega_t\cup \Omega))-\cH^{n-1}(\p^*\Omega)-\int_{\p N \cap (\Omega_t\setminus \Omega)}\cos\theta\\
		&\leq \int_{\p^*\Omega_t\setminus\Omega}\nu_{\Gamma_t}\cdot\nu_{\Gamma_t}-\int_{\p^*\Omega\cap \Omega_t}\nu_{\Gamma_t}\cdot \nu_{\p \Omega}-\int_{\p N \cap (\Omega_t\setminus \Omega)}\cos\theta\\
		&=\int_{\Omega_t\setminus \Omega} \dive(\nu_{\Gamma_t})-\int_{\p N\cap (\Omega_t\setminus\Omega)}\nu_{\Gamma_t}\cdot \nu_{\p N}+\cos\theta\\
		&\leq\int_{\Omega_t\setminus\Omega}-H_t\leq0,
	\end{align*}
	using when $t\leq t_0$ is small then $H_t>\epsilon>0$;
	and the last inequality is sharp is $\Omega_t \setminus\Omega$ has nonzero measure.
	
	A similar argument applies to show that any minimizer must be disjoint to some fixed neighborhood  $\Omega_-$ of $S_-$.
\end{proof}
We can show that the same proof applies to the warped $\theta$-bubbles provided the weight function $u$ satisfies suitable boundary assumptions.
\begin{theo}\label{ExiMinWarp}
	If $(N^n,\p N)$ is a compact Riemannian manifold, let $\theta:\p N \rightarrow \RR$ be a smooth function such that $S_{\pm}=\{x \in \p N, \cos \theta(x)=\pm 1\}$ and $S_+$ and $S_-$ are non-empty closed domains with smooth boundary in $\p N$.
	Let $u>0$ be a smooth function on $N$ with $\nabla_{\nu_{\p N}}u+uH_{\p N}>0$.
	 
	Let $\cS$ be the set of all Caccioppoli sets that contain $S_+$ and be disjoint from $S_-$, then
	\begin{equation*}
		\alpha=\inf_{\Omega \in \cS} \cA_{u}(\Omega)
	\end{equation*}
	is obtained by some $\Omega \in \cS$. If $n\leq 4$, then the minimizer $\Omega$ is a smooth submanifold with boundary that intersect $\p N$ transversally.
\end{theo}
\begin{proof}
	We apply exactly the same foliation as in Theorem \ref{MPtheta}, now we compare the following,
	
	\begin{align*}
		\cA_u(\Omega_t\cup \Omega)-\cA_u(\Omega)&=\int_{\p^*\Omega_t\setminus\Omega}u-\int_{\p^*\Omega\cap \Omega_t}u-\int_{\p N \cap (\Omega_t\setminus \Omega)}u\cos\theta\\
		&\leq\int_{\p^*\Omega_t\setminus\Omega}u\nu_{\Gamma_t}\cdot\nu_{\Gamma_t}-\int_{\p^*\Omega\cap \Omega_t} u\nu_{\Gamma_t}\cdot \nu_{\p \Omega}-\int_{\p N \cap (\Omega_t\setminus \Omega)}u\cos\theta\\
		&=\int_{\Omega_t\setminus \Omega} \dive(u\nu_{\Gamma_t})-\int_{\p N\cap (\Omega_t\setminus\Omega)}u(\nu_{\Gamma_t}\cdot \nu_{\p N}+\cos\theta)\\
		&\leq \int_{\Omega_t\setminus \Omega} \dive(u\nu_{\Gamma_t})\\
		&=\int_{\Omega_t\setminus \Omega} \nabla_{\nu_{\Gamma_t}} u+u\dive_{\Gamma_t} \nu_{\Gamma_t}\\
	\end{align*}
	Note as $t \rightarrow 0$, 
	\begin{equation*}
		\nabla_{\nu_{\Gamma_t}} u+u\dive_{\Gamma_t} \nu_{\Gamma_t} \rightarrow-\nabla_{\nu_{\p N}}u-uH_{\p N}<0.
	\end{equation*}
	Choosing $t$ small enough then we are done.
\end{proof}

We compute the first and second variation of (warped) $\theta$-bubbles.

\begin{lemm}[First Variation]
	If $\Omega$ is a smooth $\theta$-bubble in $(N^n,\p N)$ that intersect $\p N$ transversally, and $\p^*\Omega=\Sigma$, then let $\nu $ be the outward pointing unit normal of $\p \Sigma \subset \Sigma$ and $\overline{\nu}$ be the outward pointing unit normal of $\p \Sigma \subset (\overline{\Omega}\cap \p N)$, then
	\begin{equation}\label{1varTheta}
		H_{\Sigma}=0, \langle \nu,\overline{\nu}\rangle=\cos\theta.
	\end{equation}
	If $\Omega$ is a smooth warped $\theta$-bubble in $(N^n,\p N)$ that intersect $\p N$ transversally, and $\p^*\Omega=\Sigma$ and a choice of unit normal $\nu_{\Sigma}$, then	\begin{equation}\label{1varWarp}
		H_{\Sigma}=-\frac{\nabla_{\nu_{\Sigma}}u}{u}, \langle \nu,\overline{\nu}\rangle=\cos\theta.
	\end{equation}
\end{lemm}
\begin{proof}
	The proof of (\ref{1varTheta}) is the same as the case where $
	\theta$ is a constant function, as shown in \cite{ros1997stability}.
	We now prove (\ref{1varWarp}).
	Let $X_t$ be a vector field along $\Sigma_t$ with $X_t(x)\in T_x \p N$ for $x\in \p N\cap \Sigma_t$ then,
	\begin{align*}
		\frac{d}{dt}(\Vol(\Sigma_t))&=\int_{\Sigma_t} \dive_{\Sigma_t}(uX_t)-\int_{\p\Sigma_t} u\cos \theta (X_t\cdot \overline{\nu})\\
		&=\int_{\Sigma_t} \nabla_{\Sigma_t} u \cdot X_t^{\perp}+u\dive_{\Sigma_t}X_t^{\perp} +\int_{\p\Sigma_t} uX_t\cdot \nu_t-u\cos \theta (X_t\cdot \overline{\nu})\\
		&=\int_{\Sigma_t}(\nabla_{\nu_{\Sigma_t}}u+H_{\Sigma_t}u)X_t\cdot \nu_{\Sigma_t}+\int_{\p\Sigma_t} uX_t\cdot \overline{\nu_t}(\nu_t\cdot \overline{\nu_t}-\cos \theta),
	\end{align*} 
	where we used that for $t=0$, $\nu_t=(X_t\cdot \overline{\nu_t}) \overline{\nu_t}+(X_t \cdot \nu_{\p N}) \nu_{\p N}$.
	
	Setting the first variation equal to zero we obtain (\ref{1varWarp}).
\end{proof}

\begin{lemm}[Second Variation]
	If $\Omega$ is a smooth warped $\theta$-bubble in $(N^n,\p N)$ that intersect $\p N$ transversally, denote $\p^*\Omega=\Sigma$ and given a choice of unit normal $\nu_{\Sigma}$, then	let $X_t$ be a vector field along $\Sigma_t$ with $\Sigma_0=\Sigma$ and $X_t(x) \in T_x\p N$ for $x\in \p N \cap \Sigma_t$, we denote $X_0 \cdot \nu_{\Sigma}=:\phi$, then
	\begin{align*}
		\frac{d^2}{dt^2}\bigg \rvert_{t=0}\Vol(\Sigma_t)&= \int_{\Sigma} -u\phi\Delta\phi -u\phi^2(\Ric_{N}(\nu_{\Sigma},\nu_{\Sigma})+|\sff_{\Sigma}|^2)+\phi^2(\Delta_N u-\Delta_{\Sigma}u)-\phi\langle\nabla^{\Sigma}\phi,\nabla^{\Sigma}u\rangle\\
		+&\int_{\p \Sigma} u\phi \nabla_{\nu }\phi-\frac{u\phi^2}{\sin\theta}\left(\sff_{\p N}(\overline{\nu},\overline{\nu})-\cos\theta \sff_{\Sigma}(\nu,\nu )-\nabla_{\overline{\nu }} \theta \right),
	\end{align*}
	where $\nu $ is the outward pointing unit normal of $\p \Sigma \subset \Sigma$ and    $\overline{\nu }$ is the outward pointing unit normal of $\p \Sigma \subset (\overline{\Omega}\cap \p N)$.
	
	Setting $u=1$ we obtain the second variation of a $\theta$-bubble.
	
	A smooth (warped) $\theta$-bubble is called stable if for any admissible variation (if $X_t$ is a vector field along $\Sigma_t$, then $\Sigma_0=\Sigma$ and $X_t(x) \in T_x\p N$ for $x\in \p N \cap \Sigma_t$),
	\begin{equation*}
		\frac{d^2}{dt^2}\bigg \rvert_{t=0}\Vol(\Sigma_t) \geq 0.
	\end{equation*}
	
\end{lemm}

\begin{proof}
	We compute the integrand over $\Sigma$ and first show it's enough to compute the second variation over $X_t^{\perp}=(X_t \cdot \nu_{\Sigma})\nu_{\Sigma}$, let $X^T_t:=X_t-X^{\perp}_t$,
	\begin{align*}
		Q_1:&=\int_{\Sigma}\frac{d}{dt}\bigg\rvert_{t=0}(uH_{\Sigma}\cdot X_t+\nabla_{X_t^{\perp}} u)\\
		&= \int_{\Sigma}\nabla_{X^T_t}(uH_{\Sigma}\cdot X_t+\nabla_{X_t^{\perp}} u)+\nabla_{X^{\perp}_t}(uH_{\Sigma}\cdot X_t+\nabla_{X_t^{\perp}} u)\\
		&\stackrel{(*)}{=}\int_{\Sigma}\dive_{\Sigma_t} ((uH_{\Sigma}\cdot X_t+\nabla_{X_t^{\perp}} u)\cdot X^T_t)+\nabla_{X^{\perp}_t}(uH_{\Sigma}\cdot X_t+\nabla_{X_t^{\perp}} u)\\
		&\stackrel{(*)}{=}\int_{\Sigma} \nabla_{X^{\perp}_t}(uH_{\Sigma}\cdot X_t+\nabla_{X_t^{\perp}} u),
	\end{align*}
	where in the place where $(*)$ is marked we used the first variation $uH_{\Sigma}\cdot X_t+\nabla_{X_t^{\perp}} u=0$.
	
	We now compute $Q_1$ using $X_t=X_t^{\perp}=\phi_t\cdot \nu_t$ with $\nu_0=\nu_{\Sigma}$, and evaluate at $t=0$ at each step,
	\begin{align*}
		Q_1&=\int_{\Sigma}\p_t(uH_{\Sigma}+\nabla_{\nu_{\Sigma}}u)\phi\\
		&=\int_{\Sigma} \phi^2\nabla_{\nu_{ \Sigma}}u H_{\Sigma}+u\phi\p_t H_{\Sigma}+\phi^2\nabla_{\nu_{\Sigma}}\nabla_{\nu_{\Sigma}}u\\
		&=\int_{\Sigma}-u\phi\Delta\phi-u\phi^2(\Ric_N(\nu_{\Sigma},\nu_{\Sigma})+|\sff_{\Sigma}|^2)+\phi^2\nabla_{\nu_{\Sigma}}uH_{\Sigma}+\phi^2(\nabla^2u(\nu_{\Sigma},\nu_{\Sigma})+\nabla_{\nu_{\Sigma}}\nu_{\Sigma}\cdot \nabla^{\Sigma}u)\\
		&\stackrel{(**)}{=}\int_{\Sigma}-u\phi\Delta\phi-u\phi^2(\Ric_N(\nu_{\Sigma},\nu_{\Sigma})+|\sff_{\Sigma}|^2)+\phi^2(\Delta_N u-\Delta_{\Sigma}u)-\phi^2\nabla^{\Sigma}\phi\cdot \nabla^{\Sigma}u
	\end{align*}
	where in $(**)$ we used that $\Delta_{N}u-\Delta_{\Sigma}u=\nabla^2u(\nu_{\Sigma},\nu_{\Sigma})+\nabla_{\nu_{\Sigma}}uH_{\Sigma}$ and $\phi\nabla_{\nu_{\Sigma}}\nu_{\Sigma}=-\nabla^{\Sigma}\phi$.
	
	We now compute the integrand over $\p \Sigma$. One can show that it's enough to consider variations $X_t=\varphi\cdot \overline{\nu}$ with $\phi=-\varphi\sin\theta$ (at $t=0$) for this computation similar to the interior case in $Q_1$,
	\begin{align*}
		Q_2:&=\int_{\p \Sigma} u\varphi\frac{d}{dt}\bigg\rvert_{t=0}(\nu_t \cdot \overline{\nu_t}-\cos\theta)\\
		&=\int_{\p \Sigma}u\varphi(\p_t\nu_t \cdot \overline{\nu_t}+\nu_t\cdot\p_t\overline{\nu_t}+\varphi \sin\theta\nabla_{\overline{\nu}}\theta)\\
		&=\int_{\p \Sigma}\frac{-u\phi}{\sin\theta}(\p_t\nu_t \cdot \overline{\nu_t}+\nu_t\cdot\p_t\overline{\nu_t}+\varphi \sin\theta\nabla_{\overline{\nu}}\theta)
	\end{align*} 
	We compute $\p_t\nu_{t}\cdot \overline{\nu_t}$ using $\overline{\nu_t}=\cos\theta\nu_t-\sin\theta_t\nu_{\Sigma_t}$. Note only at $t=0$, $ \theta_0=\theta$ the prescribed function over $\p N$,
	\begin{align*}
		\p_t \nu_t\cdot \overline{\nu_t}&=\p_t\nu_t\cdot(\cos\theta_t\nu_t-\sin\theta_t\nu_{\Sigma_t})\\
		&=-\sin\theta_t(\p_t\nu_t\cdot \nu_{\Sigma_t})\\
		&=-\varphi\sin\theta_t(\cos\theta_t\nabla_{\nu_t}\nu_t\cdot \nu_{\Sigma_t}-\sin\theta_t\nabla_{\nu_{\Sigma_t}}\nu_t \cdot \nu_{\Sigma_t})	\\
		&= -\phi \cos\theta \sff_{\Sigma}(\nu,\nu)+\varphi\sin^2\theta(-\nu_t\cdot \nabla_{\nu_{\Sigma_t}}\nu_{\Sigma_t})\\
		&= -\phi \cos\theta \sff_{\Sigma}(\nu,\nu)+\sin\theta(-\nu\cdot \nabla^{\Sigma}\phi) \quad \text{at } t=0.
	\end{align*}
	And similarly,
	\begin{align*}
		\p_t\overline{\nu_t}\cdot \nu_{t}&=\p_t\overline{\nu_t}\cdot (\cos\theta_t\overline{\nu_t}+\sin\theta_t\nu_{\p N})\\
		&=\sin\theta (\p_t \overline{\nu_t}\cdot \nu_{\p N})\\
		&=-\varphi\sin\theta \sff_{\p N}(\overline{\nu},\overline{\nu})\\
		&=\phi \sff_{\p N}(\overline{\nu},\overline{\nu})
	\end{align*}
	In total we get,
\begin{equation*}
	Q_2=\int_{\p \Sigma} u\phi\nabla_{\nu}\phi -\frac{u\phi^2}{\sin\theta} (\sff_{\p N}(\overline{\nu},\overline{\nu})-\cos\theta \sff_{\Sigma}(\nu,\nu)-\nabla_{\overline{\nu}}\theta)
\end{equation*}
This finishes the computation.
\end{proof}

Using Schoen-Yau rearrangement we can relate the interior terms in the second variation with scalar curvature.
Using a computation in \cite{li2020polyhedron}, we can rearrange the boundary terms to relate with the mean curvature of the ambient manifold.

\begin{lemm}[Rearranged Second Variation]
	If $\Omega$ is a smooth warped $\theta$-bubble in $(N^n,\p N)$ that intersect $\p N$ transversally, denote $\p^*\Omega=\Sigma$ and given a choice of unit normal $\nu_{\Sigma}$, then	let $X_t$ be a vector field along $\Sigma_t$ with $\Sigma_0=\Sigma$ and $X_t(x) \in T_x\p N$ for $x\in \p N \cap \Sigma_t$, we denote $X_0 \cdot \nu_{\Sigma}=:\phi$, then
	\begin{align*}
		\frac{d^2}{dt^2}\bigg \rvert_{t=0}\Vol(\Sigma_t)&=  \int_{\Sigma} -\dive_{\Sigma}(u\nabla\phi)\phi -\frac{1}{2}u\phi^2(R_N-R_{\Sigma}+|\sff_{\Sigma}|^2)\\
		&+\int_{\Sigma}-\frac{|\nabla_{\nu_{\Sigma}}u|^2\phi^2}{2u}+\phi^2(\Delta_N u-\Delta_{\Sigma}u)\\
		+&\int_{\p \Sigma} u\phi \nabla_{\nu }\phi-\frac{u\phi^2}{\sin\theta}\left(H_{\p N}-\cos\theta H_{\Sigma}- \nabla_{\overline{\nu }} \theta \right)+u\phi^2\sff_{\nu}(\p \Sigma)
	\end{align*}
	where $\nu $ is the outward pointing unit normal of $\p \Sigma \subset \Sigma$,    $\overline{\nu }$ is the outward pointing unit normal of $\p \Sigma \subset (\overline{\Omega}\cap \p N)$, and $\sff_{\nu}(\p \Sigma)=-\sum_{i=1}^{n-2}\langle\nabla_{e_i}e_i,\nu\rangle$ for an orthonormal basis $(e_i)_{i=1}^{n-2}$ of $\p \Sigma$.
	
	Setting $u=1$ we obtain the rearranged second variation of a $\theta$-bubble.
\end{lemm}

\begin{proof}
We first have,
\begin{align}
	\frac{d^2}{dt^2}\bigg \rvert_{t=0}\Vol(\Sigma_t)&= \int_{\Sigma} -u\phi\Delta\phi -\frac{1}{2}u\phi^2(R_N-R_{\Sigma}+|\sff_{\Sigma}|^2+H^2_{\Sigma})\label{2varWarp}\\
		&+\int_{\Sigma}\phi^2(\Delta_N u-\Delta_{\Sigma}u)-\phi\langle\nabla^{\Sigma}\phi,\nabla^{\Sigma}u\rangle \nonumber\\
		+&\int_{\p \Sigma} u\phi \nabla_{\nu }\phi-\frac{u\phi^2}{\sin\theta}\left(H_{\p N}-\cos\theta H_{\Sigma}- \nabla_{\overline{\nu }} \theta \right)+u\phi^2\sff_{\nu}(\p \Sigma) \nonumber
\end{align}
The interior rearrangement follows from the Gauss-Codazzi equation,
	$$R_N=R_{\Sigma}+2\Ric(\nu_{\Sigma},\nu_{\Sigma})+|\sff_{\Sigma}|^2-H^2_{\Sigma}.$$
The boundary rearrangement follows from the following analogous computation of  equation (3.8) in \cite{li2020polyhedron},
	\begin{equation*}
		H_{\p N}=\sff_{\p N}(\overline{\nu},\overline{\nu})+\cos\theta H_{\Sigma}-\cos\theta \sff_{\Sigma}(\nu,\nu)+\sin\theta\sff_{\nu}(\p \Sigma).
	\end{equation*}
This gives the equality (\ref{2varWarp}).

The desired equality follows from first variation $uH_{\Sigma}+\nabla_{\nu_{\Sigma}}u=0$.
\end{proof}

\section{Inductive Process}

We now start with the inductive process of proving diameter and circumference bounds for $\theta$-bubbles. From this section onwards, any $\theta$-bubble is assume to intersect with the ambient manifold's boundary transversally.

\begin{lemm}\label{Ind2d}
	Let $(\Sigma^2,\p \Sigma)$ be a compact surface with a positive smooth function $u:\Sigma \rightarrow \RR_+$ such that,
	\begin{equation*}
		\Delta_{\Sigma} u \leq K_{\Sigma}u+\frac{|\nabla_{\Sigma} u|^2}{2u}, \quad \nabla_{\nu} u \geq a_0 u-\kappa_{\p \Sigma} u,
	\end{equation*}
	for some $a_0>0$, where $K_{\Sigma}$ is the Gauss curvature of $\Sigma$ and $k_{\p \Sigma}$ the geodesic curvature (equal to the mean curvature), and $\nu $ the outward pointing unit normal of $\p \Sigma \subset \Sigma$;
	then for any $x\in \Sigma$, 
	\begin{equation*}
		d_{\Sigma}(x,\p \Sigma)\leq \frac{2}{a_0}.
	\end{equation*}
\end{lemm}
\begin{proof}
We summarize the proof in \cite{wu2024capillary} and the reader can refer to \cite{wu2024capillary} for a more detailed proof.

Assume that there is a point $p\in \Sigma$ with $d_{\Sigma}(p,\p \Sigma)=\frac{2}{a_0}+2\epsilon$ for some $\epsilon>0$.
Let $$h(x)=\frac{2}{\frac{2}{a_0}+\epsilon-d_{\Sigma}(x,\p \Sigma)}, \quad \text{if } d_{\Sigma}(x,\p\Sigma)<\frac{2}{a_0}+\epsilon,$$
where we denote $d_{\Sigma}(x,\p \Sigma)$ the mollified distance function to the boundary and that $h^2-2|\nabla^{\Sigma}h|>0$. Let $\Omega_0$ be a fixed open set containing $\{x\in \Sigma, d_{\Sigma}(x,\p \Sigma)\geq \frac{2}{a_0}+\epsilon\}$.
	
We solve for the following warped $\mu$-bubble problem over $\Sigma$ for Caccioppoli sets in $\Sigma$ containing a fixed open neighborhood of $p$, and be disjoint from a fixed open neighborhood of $\p \Sigma$,
	\begin{equation*}
		\cA_h(\Omega)=\int_{\p \Omega}u -\int_{\Sigma} hu (\chi_{\Omega}-\chi_{\Omega_0}).
	\end{equation*}
	A maximum principle argument of first variation in \cite{wu2024capillary} shows that a smooth minimizer must exists. This is where we need the boundary mean convexity assumption $\p_{\nu} u \geq a_0u-\kappa_{\p \Sigma} u$.
	
	Then the first variation of the minimizer $\Gamma=\p \Omega$, along a smooth variation $X=\phi\nu_{\Gamma}$ is the following,
	\begin{equation}
		\frac{d}{dt}\cA_h(\Gamma)=\int_{\Gamma}\phi(uH_{\Gamma}+\nabla_{\nu_{\Gamma}}u)- hu\phi.
	\end{equation}
	And the second variation is,
	\begin{align*}
		\frac{d^2}{dt^2}\cA_h(\Gamma)&=\int_{\Gamma}\phi u(-\Delta_{\Gamma}\phi-\phi(\Ric_{\Sigma}(\nu,\nu)+|\sff_{\Gamma}|^2))+\phi^2 H_{\Sigma}\nabla_{\nu_{\Gamma}}u\\
		&+\int_{\Gamma}\phi^2 \nabla_{\nu_{\Gamma}}\nabla_{\nu_{\Gamma}}u-\phi^2\nabla_{\nu_{\Gamma}}(hu)\\
		& \stackrel{(\star 1)}{=} \int_{\Gamma}\phi u(-\Delta_{\Gamma}\phi-\phi(\frac{1}{2}R_{\Sigma}+|\sff_{\Gamma}|^2))+\phi^2 H_{\Sigma}\nabla_{\nu_{\Gamma}}u\\
		&+\int_{\Gamma}\phi^2 \nabla^2 u(\nu_{\Gamma},\nu_{\Gamma})+\phi^2\nabla^{\Gamma}u \cdot \nabla_{\nu_{\Gamma}}\nu_{\Gamma}-\phi^2\nabla_{\nu_{\Gamma}}(hu)\\
		& \stackrel{(\star 2)}{=}\int_{\Gamma}\phi (-\dive_{\Gamma}(u\nabla_{\Gamma}\phi)-u\phi(\frac{1}{2}R_{\Sigma}+|\sff_{\Gamma}|^2))+\phi^2(\Delta_{\Sigma}u-\Delta_{\Gamma}u)-\phi^2\nabla_{\nu_{\Gamma}}(hu)\\
		& \stackrel{(\star 3)}{\leq} \int_{\Gamma} \phi(-\dive_{\Gamma}(u\nabla_{\Gamma}\phi))+\frac{|\nabla_{^{\Sigma}}u|^2}{2u}\phi^2-u\phi^2|\sff_{\Gamma}|^2-\phi^2\Delta_{\Gamma}u-\phi^2\nabla_{\nu_{\Gamma}}(hu)\\
		& \stackrel{(\star 4)}{\leq}\int_{\Gamma} \phi(-\dive_{\Gamma}(u\nabla_{\Gamma}\phi))+\frac{|\nabla_{^{\Sigma}}u|^2}{2u}\phi^2-\frac{u\phi^2}{2}(h-\frac{\nabla_{\nu_{\Gamma}} u}{u})^2-\phi^2\Delta_{\Gamma}u-\phi^2\nabla_{\nu_{\Gamma}}(hu)\\
		&=\int_{\Gamma}u(\nabla_{\Gamma}\phi)^2+\frac{|\nabla_{^{\Sigma}}u|^2}{2u}\phi^2-\frac{u\phi h^2}{2}-\frac{(\nabla_{\nu_{\Gamma}}u)^2}{2u}-\phi^2\Delta_{\Gamma}u-\phi^2u\nabla_{\nu_{\Gamma}}h\\
		&=\int_{\Gamma}\frac{|\nabla_{\Gamma}u|^2}{2u}\phi^2-\phi^2\Delta_{\Gamma}u-\frac{\phi^2u}{2}(h^2+2\nabla_{\nu_{\Gamma}}h)+u(\nabla_{\Gamma}\phi)^2\\
		&\stackrel{(\star 5)}{<}\int_{\Gamma}\frac{|\nabla_{\Gamma}u|^2}{2u}\phi^2-\phi^2\Delta_{\Gamma}u+u(\nabla_{\Gamma}\phi)^2\\
		&=\int_{\Gamma}\frac{|\nabla_{\Gamma}u|^2}{2u}\phi^2+\nabla_{\Gamma}u\nabla_{\Gamma}\phi^2+u(\nabla_{\Gamma}\phi)^2.
	\end{align*}
	The computation of warped $\mu$-bubble here works in general dimensions with small adaptions. 
	In $(\star 1)$ we used for surfaces the Ricci curvature is one half of the scalar curvature; in general dimensions one can use the Gauss-Codazzi equation.
	In $(\star 2)$ we used that for a normal variation with speed $\phi$, $\p_t \nu_{\Gamma}=-\nabla^{\Gamma}\phi$. In $(\star 3)$ we used the interior NNSC assumption $\Delta_{\Sigma} u \leq K_{\Sigma}u+\frac{|\nabla_{\Sigma} u|^2}{2u}$.
	In $(\star 4)$ we used the first variation $|\sff_{\Gamma}|^2=(H_{\Gamma})^2=(h-\frac{\nabla_{\nu_{\Gamma}}u}{u})^2$ and $|\sff_{\Gamma}|^2 \geq \frac{|\sff_{\Gamma}|^2}{2}$. Lastly in $(\star 5)$ we used $h^2-2|\nabla^{\Sigma}h|>0$.

	We now choose $\phi^2u=1$ to get,
	\begin{equation*}
		\frac{|\nabla_{\Gamma}u|^2}{2u}\phi^2+\nabla_{\Gamma}u\nabla_{\Gamma}\phi^2+u(\nabla_{\Gamma}\phi)^2=\frac{|\nabla_{\Gamma} u|^2}{2u^2}-\frac{|\nabla_{\Gamma}u|^2}{u^2}+\frac{|\nabla_{\Gamma}u|^2}{4u^2}=-\frac{|\nabla_{\Gamma}u|^2}{4u^2},
	\end{equation*}
	which gives us a contradiction to the stability inequality.
\end{proof}

\begin{lemm}\label{Ind3d}
	Let $(\Sigma^2, \p \Sigma)$ be a stable $u$-warped $\theta$-bubble in a 3-manifold $(M^3,\p M)$ with,
	\begin{equation*}
		\Delta_M u\leq \frac{R_M}{2}u+\frac{|\nabla^M u|^2}{2u},
	\end{equation*}
	and over the boundary $\p \Sigma$ we have,
	\begin{equation}\label{2dWarpBC}
		 \nabla_{\nu_{\p M}}u+u(H_{\p M}-\nabla_{\overline{\nu}}\theta)\geq (\sin\theta) a_0u>0
	\end{equation}
	then we have $|\p\Sigma|\leq \frac{2\pi}{a_0}$ and $d_{\Sigma}(x,\p\Sigma)\leq \frac{2}{a_0}$ for all $x\in \Sigma$. 
\end{lemm}

\begin{proof}
	The rearranged second variation over $\Sigma$ gives,
	\begin{align*}
		0&\leq  \int_{\Sigma} -\dive_{\Sigma}(u\nabla\phi)\phi -\frac{1}{2}u\phi^2(R_M-R_{\Sigma}+|\sff_{\Sigma}|^2)\\
		&+\int_{\Sigma}-\frac{|\nabla_{\nu_{\Sigma}}u|^2\phi^2}{2u}+\phi^2(\Delta_M u-\Delta_{\Sigma}u)\\
		+&\int_{\p \Sigma} u\phi \nabla_{\nu }\phi-\frac{u\phi^2}{\sin\theta}\left(H_{\p M}-\cos\theta H_{\Sigma}- \nabla_{\overline{\nu }} \theta \right)+u\phi^2\sff_{\nu}(\p \Sigma)\\
		&\leq \int_{\Sigma} -\dive_{\Sigma}(u\nabla\phi)\phi +\frac{1}{2}u\phi^2(R_{\Sigma}-|\sff_{\Sigma}|^2) +\frac{|\nabla^{{\Sigma}}u|^2\phi^2}{2u}+\phi^2(-\Delta_{\Sigma}u)\\
		+&\int_{\p \Sigma} u\phi \nabla_{\nu }\phi-\frac{u\phi^2}{\sin\theta}\left(H_{\p M}-\cos\theta H_{\Sigma}- \nabla_{\overline{\nu }} \theta \right)+u\phi^2\sff_{\nu}(\p \Sigma)
	\end{align*}
	Then we obtain that for some choice of $\phi>0$, 
	\begin{align*}
		\dive_{\Sigma}(u\nabla \phi)&\leq \frac{R_{\Sigma}}{2}u\phi-\phi\Delta_{\Sigma}u+\frac{|\nabla^{\Sigma}u|^2\phi}{2u}, \\
		\nabla_{\nu}\phi&= \frac{\phi}{\sin\theta}(H_{\p M}-\nabla_{\bar{\nu}}\phi-\cos\theta H_{\Sigma})-\phi\sff_{\nu}(\p \Sigma)
	\end{align*}
	Let $f=u\phi$, then
	\begin{align*}
		\Delta_{\Sigma}f&=\dive_{\Sigma}(u\nabla_{\Sigma}\phi)+\phi \Delta_{\Sigma}u+\nabla^{\Sigma}u\cdot\nabla^{\Sigma}\phi\\
		&\leq \frac{R_{\Sigma}}{2}u\phi+ \nabla^{\Sigma}u\cdot \nabla^{\Sigma}\phi + \frac{|\nabla^{\Sigma}u|^2\phi}{2u}\\
		&\leq \frac{R_{\Sigma}}{2}u\phi+\frac{|\nabla_{\Sigma}(u\phi)|^2}{2u\phi}=\frac{R_{\Sigma}f}{2}+\frac{|\nabla_{\Sigma} f|^2}{2f}.
	\end{align*}
	We now check  the boundary condition of $f$. We first claim that for $\nu $ the outward pointing unit normal of $\p \Sigma \subset \Sigma$
	\begin{equation}\label{BCclaim}
		\nabla_{\nu}u+\frac{u}{\sin\theta}(H_{\p M}-\nabla_{\overline{\nu}}\theta-\cos\theta H_{\Sigma})\geq a_0 u >0
	\end{equation}
	Using (\ref{BCclaim}) we get,
	\begin{align*}
		\p_{\nu}f&=\phi\p_{\nu}u+u\p_{\nu}\phi\\
		&= \phi\p_{\nu}u+\frac{u\phi}{\sin\theta}(H_{\p M}-\nabla_{\bar{\nu}}\theta-\cos\theta H_{\Sigma})-u\phi\sff_{\nu}(\p \Sigma)\\
		&\geq a_0(u\phi)-\sff_{\nu}(\p \Sigma)(u\phi)=a_0f-\sff_{\nu}(\p \Sigma)f
	\end{align*}
	This implies that $d_{\Sigma}(x,\p \Sigma)\leq \frac{2}{a_0}$ for all $x\in \Sigma$ using Lemma \ref{Ind2d}.
	
	The proof of (\ref{BCclaim}) uses the relationship of different normal vectors along a $\theta$-bubble.  
	\begin{align*}
		\nu &:=\nu_{\p \Sigma}=\cos\theta \bar{\nu}+\sin\theta \nu_{\p M}, \quad &\nu_{\Sigma}&=-\sin\theta \bar{\nu}+\cos\theta\nu_{\p M}
		\\
		\nabla_{\nu_{\Sigma}}u&=-\sin\theta \nabla_{\bar{\nu}}u+\cos\theta\nabla_{\nu_{\p M}}u, \quad &\nabla_{\nu}u&=\cos\theta \nabla_{\bar{\nu}}u+\sin\theta\nabla_{\nu_{\p M}}u
	\end{align*}
	This implies
	\begin{align*}
		&\nabla_{\nu}u+\frac{u}{\sin\theta}(H_{\p M}-\nabla_{\overline{\nu}}\theta-\cos\theta H_{\Sigma})\\
		&=\nabla_{\nu}u-\frac{\cos\theta}{\sin\theta}(-\nabla_{\nu_{\Sigma}}u)+\frac{u}{\sin\theta}(H_{\p M}-\nabla_{\overline{\nu}}\theta)\\
		&=\cos\theta\nabla_{\bar{\nu}}u+\sin\theta\nabla_{\nu_{\p M}}u-\cos\theta\nabla_{\bar{\nu}}u+\frac{\cos^2\theta}{\sin\theta}\nabla_{\nu_{\p M}}u+\frac{u}{\sin\theta}(H_{\p M}-\nabla_{\overline{\nu}}\theta)\\
		&=\frac{1}{\sin\theta} \nabla_{\nu_{\p M}}u+\frac{u}{\sin\theta}(H_{\p M}-\nabla_{\overline{\nu}}\theta)\geq a_0u>0.
	\end{align*}

	Finally, we can use Gauss-Bonnet in the second variation formula. When choosing $u\phi^2=1$, one can simplify to get,
	\begin{align*}
		0& \leq \int_{\Sigma} u(\nabla^{\Sigma}u^{-\frac{1}{2}})^2+\frac{1}{2}R_{\Sigma} +\frac{|\nabla^{\Sigma}u|^2}{2u^2}-\frac{\Delta_{\Sigma} u}{u}+\int_{\p\Sigma}-\frac{1}{\sin\theta}(H_{\p M-\cos\theta H_{\Sigma}}-\nabla_{\bar{\nu}}\theta)+k_{\p\Sigma} \\
		&= \int_{\Sigma} -\frac{1}{4}u^{-1}|\nabla^{\Sigma} u|^2 +\int_{\Sigma} \frac{R_{\Sigma}}{2}+\int_{\p\Sigma} k_{\p\Sigma}+\int_{\p\Sigma}-\frac{\nabla_{\nu}u}{u}-\frac{1}{\sin\theta}\left(H_{\p M}-\cos\theta H_{\Sigma}- \nabla_{\overline{\nu }} \theta \right)\\
		&\leq \int_{\Sigma} -\frac{1}{4}u^{-1}|\nabla^{\Sigma} u|^2 +2\pi\chi_{\Sigma}-\int_{\p \Sigma}a_0,
	\end{align*}
	Since $\chi_{\Sigma}$ is a $\theta$-bubble ($\p\Sigma \neq \emptyset$), $\chi_{\Sigma}\leq 1$ implies  $|\p\Sigma|\leq \frac{2\pi}{a_0}$.
\end{proof}

\begin{coro}\label{PutInBubble}
	Consider $(M^3,\p M)\hookrightarrow (X^4,\p X)$  a FMBH and $X$ has $$R_{X}\geq 0, \quad H_{\p X}\geq H_0>0.$$ 
	Let the diameter of $\p M$ (with respect to $d_{\p M}(\cdot,\cdot)$ under the induced metric) be larger than some $d_0>0$, and assume 
	\begin{equation*}
		H_0-\frac{\pi}{d_0}=a_0>0,
	\end{equation*}
	then we can find a stable warped $\theta$-bubble $(\Sigma^2,\p \Sigma)$ such that,
	\begin{equation*}
		|\p \Sigma| \leq \frac{2\pi}{a_0}, \quad d_{\Sigma}(x,\p \Sigma)\leq \frac{2}{a_0}.
	\end{equation*}
\end{coro}
\begin{proof}
	Using the stability inequality of $(M,\p M)$, we can find a positive $u: M \rightarrow \RR_+$
	\begin{equation*}
		\Delta_{M} u\leq \frac{1}{2}R_M, \quad \nabla_{\nu_{\p M}}u=\sff_{\p X}(\nu_{M},\nu_{M})u.
	\end{equation*}
	Now using the diameter of $\p M$ is at least $d_0$, we want to use the existence of smooth minimizer in Theorem \ref{ExiMinWarp}, for $S_{\pm}=\{x\in \p M, \cos\theta(x)=\pm 1\}$, and we can assume that $|\nabla^{\p M}\theta|\leq \frac{\pi}{d_0}$. We check that in this case,
	\begin{equation*}
		\nabla_{\nu_{\p M}}u+uH_{\p M}=uH_{\p X}\geq uH_0>0.
	\end{equation*}

	We now minimize the following functional as in Theorem \ref{ExiMinWarp} to get a stable $u$-warped $\theta$-bubble,
	\begin{equation*}
		\cA_u(\Omega)=\int_{\p^* \Omega} u -\int_{\p M\cap \Omega}u\cos\theta.
	\end{equation*}
	We remark that now the ambient manifold $M$ is not compact (in Theorem \ref{ExiMinWarp} we required compactness). We can use a similar adaptation to non-compact manifolds, as in the proof of Theorem 1.4 and Remark 3.2  in \cite{wu2024capillary}. This is because the condition $H_0-\frac{\pi}{d_0}$ will constrain the minimizer $ \Sigma=\p \Omega$ to lie close to its boundary $\p \Sigma$, which was prescribed to lie in a bounded region. We refer  the reader to \cite{wu2024capillary} for more details.
	
	We now check that condition (\ref{2dWarpBC}) is satisfied. 
	\begin{align*}
		\nabla_{\nu_{\p M}}u+u(H_{\p M}-\nabla_{\overline{\nu}}\theta)&= \sff_{\p X}(\nu_M,\nu_M)u+u(H_{\p M}-\nabla_{\overline{\nu}}\theta) \\
		&= u(H_{\p X}-\nabla_{\bar{\nu}}\theta)\geq u(H_0-\frac{\pi}{d_0})=a_0u\geq (\sin\theta)a_0u>0
	\end{align*}
	
	Now Lemma \ref{Ind3d} applies to give the desired bounds for the $\theta$-bubble.
\end{proof}

Finally we summarize the inductive procedure for smooth stable warped $\theta$-bubble in a manifold of general dimension with an analogous NNSC and strictly mean convexity assumption. The following lemma can be viewed as the PSC equivalent ``inheritance'' phenomenon for manifolds with NNSC and mean convex boundary. The proof is the same as Lemma \ref{Ind3d}.
\begin{lemm}\label{inherit}
	If $(\Sigma^n,\p \Sigma)\hookrightarrow (M^{n+1},\p M)$ is a smooth stable $u$-warped $\theta$-bubble, with $u>0$ a smooth function on $M$, and for some $a_0>0$,
	\begin{align*}
		\Delta_M u&\leq \frac{R_M}{2}u+\frac{|\nabla^M u|^2}{2u}\\
		\nabla_{\nu_{\p M}}u&\geq (\sin\theta)a_0 u-\frac{u}{\sin\theta}(H_{\p M}-\nabla_{\bar{\nu}}\theta)
	\end{align*}
	then $(\Sigma^n,\p \Sigma)$ has for some $f>0$,
	\begin{align*}
		\Delta_{\Sigma}f&\leq \frac{R_{\Sigma}}{2}f+\frac{|\nabla^{\Sigma}f|^2}{2f}\\
		\nabla_{\nu_{\p \Sigma}} f&\geq a_0 f-\sff_{\nu}(\p \Sigma)f
	\end{align*}
\end{lemm}

\section{Parabolicity and Stability}

In \cite{cao1997structure}, the authors proved that a complete non-compact two-sided stable minimal hypersurfaces in $\RR^{n+1}$ for $n\geq 3$ must have only one end. The use of Michael-Simon-Sobolev inequality for minimal hypersurfaces in $\RR^{n+1}$ is crucial since this implies that any end of a stable minimal hypersurfaces must be non-parabolic. And having two non-parabolic ends together with the Schoen-Yau rearranged stability inequality using the Bochner formula, this implies that the stable minimal hypersurface must admit a non-constant harmonic function with finite energy and must have finite volume, a contradiction to the monotonicity formula.

In \cite{chodosh2024complete}, the authors proved that under the assumption of non-negative 2 intermediate Ricci curvature of a 4 dimensional ambient manifold, a stable minimal hypersurface with infinite volume can have at most one non-parabolic end. This was extended to the FBMH case in \cite{wu2023free}.

In section \ref{secRigidity}, we will show that if $M$ is a simply connected stable minimal hypersurface in $(X^4,\p X)$ as in our main theorem, then $\p M$ must be connected and have an end in the only non-parabolic end of $M$. In this section we prove some auxiliary results.  

We first prove a generalized case to Theorem 5.2 in \cite{wu2023free}.
\begin{lemm}\label{BochnerStabNeu}
	Let $(X^4,\p X)$ be a complete Riemannian manifold with $\Ric^X_2 \geq 0$ and $(M^3,\p M)$ a FBMH in $X$. If either
	\begin{itemize}
		\item $u$ is a smooth harmonic function on $M$ with Neumann boundary condition and $ \sff_2^{\p X}\geq 0$,
		\item or $u$ is a smooth harmonic function on $M$ with Dirichlet boundary condition and $H_{\p X}\geq 0$,
	\end{itemize}
	then
	\begin{equation*} 
		\int_M \phi^2( \frac{1}{3}|\sff|^2|\nabla u|^2+\frac{1}{2}|\nabla|\nabla u||^2)\leq \int_M |\nabla \phi|^2 |\nabla u|^2.
	\end{equation*}
\end{lemm}

\begin{proof}

	 The proof of the Neumann case is in \cite{wu2023free}, Theorem 5.2.
	 
	For Dirichlet case, the same proof of applies to get the following (without assuming any boundary condition),
	\begin{align}\label{Lem4dot2}
		&\int_M \phi^2( \frac{1}{3}|\sff|^2|\nabla u|^2+\frac{1}{2}|\nabla|\nabla u||^2)\\
		\leq & \int_M |\nabla \phi|^2 |\nabla u|^2+\int_{\p M}\phi^2(\frac{1}{2}\nabla_{\nu_{\p M}}|\nabla u|^2-\sff_{\p X}(\nu_{M},\nu_M)|\nabla u|^2).\nonumber
	\end{align}
	We have  $0=\Delta_M u=\Delta_{\p M}u+\nabla^2 u(\nu_{\p M},\nu_{\p M})+H_{\p M} \nabla_{\nu_{\p M}}u$.  The Dirichlet boundary condition implies $\Delta_{\p M}u=0$ and $\nabla u=\nabla_{\nu_{\p M}} u \cdot \nu_{\p M}$,
	\begin{align*}
		\frac{1}{2}\nabla_{\nu_{\p M}} |\nabla u|^2&=\nabla_{\nu_{\p M}} \nabla u \cdot \nabla u\\
		&=(\nabla_{\nu_{\p M}}u)\nabla^2 u(\nu_{\p M},\nu_{\p M})\\
		&=-(\nabla_{\nu_{\p M}} u)^2 H_{\p M}=-H_{\p M}|\nabla u|^2.
	\end{align*}
	Plug in the last computation into (\ref{Lem4dot2}) and use $H_{\p X}\geq 0$ then we obtain the desired inequality.
\end{proof}

\begin{rema}\label{DirNeu}
	Note that $\sff^{\p X}_2 \geq 0$ implies $H_{\p X}\geq 0$.
	If $\p M$ has more than one component and $\p M$ is 2-convex, the same result as in Lemma \ref{BochnerStabNeu} holds if one assumes Dirichlet boundary or Neumann boundary conditions for the harmonic function on different components.
\end{rema}

We first recall some definitions and lemmas about parabolicity. Details about this notion can be found in  \cite{li2004lectures}, \cite{chodosh2024complete} and \cite{wu2023free}. For regularity reasons, we assume here that each end $E \subset (M \setminus K)$ for some compact set $K$, if has corner points (where $\p_0 E:=\p M\cap E$ and $\p_1 E:= \overline{E} \cap K$ intersect), then the interior corner has small angles (no more than $\frac{\pi}{8}$). See \cite{wu2023free} section 4 for more details on this.

\begin{defi}[\cite{wu2023free}, Definition 4.4]
	A component $E \subset (M\setminus K)$ for some compact set $K \subset M$ of a Riemannian manifold $M$ is non-parabolic if there is a non-constant positive harmonic function $f : E \rightarrow (0,1]$, and $\p_{\nu_{\p M}} f \rvert_{\p_0E}=0, f\rvert_{\p_1 E}=1$. Otherwise $E$ is parabolic.
\end{defi}

\begin{lemm}[\cite{wu2023free}, Theorem 4.12]
	If $E$ is non-parabolic, then there is a unique function $f:E \rightarrow (0,1]$ with Neumann boundary condition on $\p_0 E$ and Dirichlet on $\p_1 E$, such that if $g:E\rightarrow (0,1]$ is also a harmonic function with Neumann boundary condition on $\p_0 E$ and Dirichlet boundary condition on $\p_1 E$, then $f\leq g$. We call this the minimal barrier function over the non-parabolic component $E$. Furthermore, we can assume the minimal barrier function has finite Dirichlet energy.
\end{lemm}

\begin{lemm}[\cite{wu2023free}, Lemma 4.10]\label{NPinherit}
	If $K\subset K'$ are two compact sets in $M$ and $E \subset (M\setminus K)$ is a non-parabolic component, then there is a component $E' \subset (M \setminus K')$ and $E' \subset E$, such that $E'$ is non-parabolic.
\end{lemm}

Using a proof similar to Theorem 4.12 in \cite{wu2023free} we have the following.
\begin{lemm}\label{NPherit}
	If $K \subset K'$ are two compact sets in $M$ and $E'\subset (M \setminus K')$ is a non-parabolic component, then the component $E \subset (M \setminus K)$ that contains $E'$ must be  non-parabolic.
	
	Equivalently, if $K \subset K'$ are two compact sets in $M$, and $P \subset (M \setminus K)$ is parabolic, then each component of $(M \setminus K') \cap P$ must be parabolic.
\end{lemm}

\begin{defi}
	We say a Riemannian manifold $M$ is parabolic if there is a point $x\in M$ and a small geodesic ball $B_{r}(x)$ for some $r>0$ such that the connected set $M\setminus B_r(x)$ is parabolic. 
	Otherwise we say $M$ is non-parabolic.
\end{defi}
We note that this definition is independent of the choice of $x\in M$ and $r>0$. Indeed, if the connected set $M \setminus B_{r_1}(x_1)$ is parabolic, then so is the connected set $M\setminus B_{r_2}(x_2)$. Because if $M \setminus B_{r_2}(x_2)$ is non-parabolic, take $R$ large so that $B_{r_1}(x_1) \subset B_{R}(x_2)$, then $M\setminus B_R(x_2)$ must have a non-parabolic component by Lemma \ref{NPinherit}, which then again implies that the set $M\setminus B_{r_2}(x_2)$ is non-parabolic by Lemma \ref{NPherit}.

	We say that $M$ has at most one non-parabolic end, if for any compact set $K \subset M$, there is at most one non-parabolic component of $M \setminus K$. If $M$ is non-parabolic, then $M$ has exactly one non-parabolic end.

We recall Theorem 5.3 in \cite{wu2023free}, that a stable FMBH in a 4-manifold with non-negative ``2-intermediate Ricci curvature'' and ``2-convex'' boundary can only have at most one non-parabolic end if its volume is infinite.
We further prove each component of $\p M$ must be non-compact.

\begin{theo}\label{NcptBound}
	If $(M^3,\p M) \rightarrow (X^4,\p X)$ is a FBMH and $\Ric^X_2\geq 0, \sff^{\p X}_2\geq 0$, then
	\begin{itemize}
		\item either $\Vol(M)<\infty$
		\item or $\Vol(M)=\infty$ and $M$ has at most 1 non-parabolic end.
	\end{itemize}
	If $M$ is non-parabolic and $\Vol(M)=\infty$, then each component of $\p M$ must be non-compact.
\end{theo}

\begin{proof}
	The proof that $M$ with infinite volume must have at most 1 non-parabolic end is the same as Theorem 5.3 in \cite{wu2023free}. We briefly summarize it here.
	
	Given any compact set $K\subset M$, if $E_i(i\in\{1,2\})$ is non-parabolic in $M\setminus K$, then on each $E_i$ there is a minimal barrier function $f_i$ with finite Dirichlet energy. Then using these two barrier functions one can construct a non-constant harmonic function on $M$ with finite Dirichlet energy and Neumann boundary condition on $\p M$.
	Using a linear cut-off function for $\phi$ in Lemma \ref{BochnerStabNeu}, we can show that this implies $\Vol(M)$ must be finite.
	
	We now show that if $M$ is non-parabolic, then either $M$ has finite volume, or each component of $\p M$ must be non-compact. 
	
	Assume $\p_1 M$ is a compact component of $\p M$ and denote $\p_0 M:=\p M \setminus \p_1 M$ (this could be the empty set).
	
	We minimize Dirichlet energy on $B_{R}(x)\supset \p_1 M$ for some $x\in M$ and $R>0$, among functions $f$ such that,
	\begin{equation*}
		f \rvert_{\p_1 M}=1,  f\rvert_{\p' B_R(x)}=0, \nabla_{\nu_{\p M}} f\rvert_{\p_0 M}=0.
	\end{equation*}
	
	As $R\rightarrow \infty$, the minimizer $f_R$ converges in $C^{2}_{\text{loc}}(M)$ to a  harmonic function $f_{\infty}$, and $(f_{\infty})\rvert_{\p_1 M}=1, \nabla_{\nu_{\p M}}f_{\infty} \rvert_{\p_0 M}=0$. 
	
	Let $g_{\infty}$ be the minimal barrier function on the non-parabolic end $E\subset (M \setminus B_r(x))$ for a fixed $r>0$ so that $\p_1M \subset B_r(x)$. Then $f_{R} \rvert_{E\cap B_R(x)} < g_{\infty} \rvert_{E\cap B_R(x)}$, passing to the limit we get $f_{\infty}\rvert_E \leq g_{\infty}\rvert_E$, which implies $f_{\infty}$ is non-constant.
	
	The minimizing solution $f_{R}$ has decreasing Dirichlet energy as $R\rightarrow \infty$. This implies that $f_{\infty}$ is a non-constant harmonic function with finite energy and Dirichlet boundary condition on $\p_1 M$ and Neumann boundary condition on $\p_0 M$. By Remark \ref{DirNeu} and Lemma \ref{BochnerStabNeu} we get that $\Vol(M)$ must be finite, a contradiction.
\end{proof}
We now prove each component of the boundary $\p M$ must has an end in the non-parabolic end of $M$.

\begin{lemm}\label{boundInNP}
	Consider $(M^3,\p M)\hookrightarrow (X^4,\p X)$ a FBMH and $\Ric^X_2\geq 0, \sff^{\p X}_2 \geq 0$.
	Assume $(M,\p M)$ is non-parabolic, $(C_k)_{k\in \NN}$ is a sequence of compact sets and $E_k \subset (M \setminus C_k)$ is a sequence of nested non-parabolic components,  if $\Vol(M)=\infty$ then any connected component $\Gamma$ of $\p M$, must have $(\Gamma\setminus C_k) \cap E_k \neq \emptyset$ for any $k\in \NN$.
\end{lemm}

\begin{proof}
	Assume $\Gamma$ is a component of $\p M$ and there is $k_0\in \NN$ such that $(\Gamma\setminus C_{k_0})\cap E_{k_0}=\emptyset$, then for any $k\geq k_0$, $(\Gamma \setminus C_k)\cap E_k =\emptyset$. Note that $\Gamma\setminus C_k\neq \emptyset$ since $\Gamma$ is non-compact by  Lemma \ref{NcptBound}.
	
	We now minimize Dirichlet energy on $C_k (k\geq k_0)$ among functions satisfying the following,
	\begin{equation*}
		f\rvert_{\Gamma\cap C_k}=1, \nabla_{\nu_{\p M}}f\rvert_{C_k\cap(\p M \setminus \Gamma)}=0, f\rvert_{\p' C_k\cap E_k }=0, f\rvert_{\p'C_k \setminus E_k}=1.
	\end{equation*}
	Denote the minimizer as $f_k: C_k \rightarrow [0,1]$, and the minimal barrier function on $E_{k_0}$ as $f_0$.
	Then using extension by constants, $\int_{C_k}|\nabla f_k|^2$ is non-increasing. By maximum principle, $f_k\rvert_{E_{k_0} \cap C_k}\leq f_0 \rvert_{E_{k_0}\cap C_k}$, and passing to the limit we have $f_k$ converges in $C^{2}_{\text{loc}}(M)$ to a harmonic function $f_{\infty}: M \rightarrow (0,1]$ such that $f_{\infty}\rvert_{\Gamma}=1, \nabla_{\nu_{\p M}}f_{\infty}\rvert_{\p M \setminus \Gamma}=0$, and $f_{\infty}$ is non-constant, because $f_{\infty}\rvert_{E_{k_0}} \leq f_0\rvert_{E_{k_0}}$.
	
	In total we get a non-constant harmonic function with finite Dirichlet energy, using Remark \ref{DirNeu} we get that the volume of $M$ must be finite.
\end{proof}

\section{Proof of almost linear volume growth and Rigidity}\label{secRigidity}

We can now start to prove Theorem \ref{rigidity}. We can first pass to the universal cover of $M$, in this section, we assume that the FBMH $M$ is simply connected. We assume all the assumptions of Theorem \ref{rigidity}.  By scaling we may assume that $H_{\p X}\geq 2$.

We start with a lemma that shows that $\p M$
 must be connected if $M$ is simply connected. 
 
Recall we use the notation $\p M$ to denote the boundary of a continuous manifold $(M,\p M)$, and if $\Omega$ is an open subset of $M$, we denote the topological boundary as $\p'\Omega$. 
 \begin{lemm}\label{connBound}
 	Assume $(M,\p M)$ as in Theorem \ref{rigidity} is simply connected, each component of $\p M$ is non-compact (by Lemma \ref{NcptBound}), and each component of $\p M$ must have an end in the end $(E_k)_{k\in \NN}$ in the sense of Lemma \ref{boundInNP} (we do not assume $(E_k)$ is non-parabolic), then $\p M$ is connected. 
 \end{lemm}
 
 \begin{proof}
 Now let $\p_1 M, \p_2 M$ be two non-compact components. We may take the compact set $C_k:=B_{3k\pi}(x)$ for some $x\in M$, and $E_k$ is a nested sequence of sets in $M\setminus C_k$. Then using Corollary \ref{PutInBubble} for $H_0=2, d_0=\pi, a_0=1$, we can find a $\theta$-bubble $\Sigma_k=\p'\Omega_k$ inside $C_{k+1}\setminus C_k$, with each component $\Sigma^{\alpha}_k$ of $\Sigma_k$  having $|\p \Sigma^{\alpha}_k|\leq 2\pi$ and $d_{\Sigma_k}(z,\p \Sigma^{\alpha}_k)\leq 2$ for  $z\in \Sigma^{\alpha}_k$.
 
 By simply connectedness of $M$, $\p' E_k$ must be connected for any $k\in \NN$ and $(E_k \setminus E_{k+1}) \cap C_{k+1}=: M_k$ must also be connected (see \cite{chodosh2024complete}, \cite{wu2023free},\cite{wu2024capillary}).
 
 Now take a point $p_1\in \p_1 M$ and $p_2\in \p_2 M$, for $k$ larger than some $k_0$ we may assume that $p_1, p_2 \in C_k$. Since both $\p_1 M$ and $\p_2 M$ are non-compact and has an end in $(E_k)_{k\geq k_0}$, we must have $\p' E_k \cap \p_{i} M\neq \emptyset$ for $k\geq k_0$ and $i\in \{1,2\}$. That is, any path in  $\p_i M$ connecting $p_i$ to some point $q_i\in E_k \cap \p_i M$ must intersect $ \p' E_k \cap \p_i M$.  If $q_i \in \p' E_{k+1}$, then a path $\overline{p_i q_i}$ in $\p_i M$ must intersect $\p \Sigma_k$. Since each component of  $\Sigma_k$ is a disk, we must have that $\overline{p_1q_1}$ intersect $\p \Sigma^{\alpha}_k$, and $\overline{p_2q_2}$ intersect $\p \Sigma^{\beta}_k$ for $\alpha \neq \beta$. We can now connect $q_1, q_2$ with a path in $\p' E_{k+1}$, and connect $p_1, p_2$  with a path in $C_{k_0}$. This gives us a loop in $M$ with non-trivial intersection to a disk, contradicting simply connectedness. 

 In total, we have shown that at most one component of $\p M$ can have an end in $(E_k)_{k\in \NN}$, so $\p M$ is connected.
 \end{proof}

We now assume that $(E_k)_{k\in \NN}$ is the non-parabolic end of $M$ with respect to $C_k:=B_{3k\pi}(x)$ for some $x\in \p M$. If $M$ is parabolic then $E_k=\emptyset$. Assume $\p M$ is connected. We denote $P_k:= (M \setminus C_k) \setminus E_k$, and each component of $P_k$ is parabolic. Denote $M_k:=(E_{k}\setminus E_{k+1})\cap C_{k+1}$.
We now have the following decomposition of $M$,
\begin{align*}
	M=B_{3\pi}(x) \cup E_1 \cup P_1&=B_{3\pi}(x)\cup P_1\cup (M_q\cup P_2 \cup E_2)\\
	&=B_{3\pi}(x)\cup (\cup_{k=1}^m P_k) \cup (\cup_{k=1}^m M_k) \cup E_{k+1}
\end{align*}

\begin{proof}[Proof of Theorem \ref{rigidity}]
	Using Corollary \ref{PutInBubble} we can find a $\theta$-bubble $\Sigma_k=\p'\Omega_k$ inside $C_{k+1}\setminus C_k$ with $|\p \Sigma^{\alpha}_k | \leq 2\pi, d_{\Sigma_k}(z,\p\Sigma^{\alpha}_k) \leq 2$ for any component $\Sigma^{\alpha}_k$ of $\Sigma_k$ and $z\in \Sigma^{\alpha}_k$.  In particular, we have the diameter of each $\Sigma^{\alpha}_k$ is at most $\pi +4$.
	
	\textit{Claim:} there is a unique component $\Sigma^{\alpha}_k$ with $\p \Sigma^{\alpha}_k \neq \emptyset$ separating $\p' E_k$ and $\p' E_{k+1}$, that is, any path from $\p' E_k$ to $\p' E_{k+1}$ must intersect $\Sigma^{\alpha}_k$.
	
	Take $q_{k+1} \in \p' E_{k+1} \cap \p M$, connect $x\in \p M$ and $q_{k+1}$ with a path in $\p M$, then $\overline{x q_{k+1}}$ must intersect $\p'E_{k}\cap \p M$ at some point $q_k$, and must intersect some component $\Sigma^{\alpha}_k$ with $\p \Sigma^{\alpha}_k \neq \emptyset$. Now if there is another path $\gamma$ from $\p'E_{k}$ to $\p' E_{k+1}$ that is disjoint from $\Sigma^{\alpha}_k$, as in Lemma \ref{connBound}, since $\p'E_k$ is connected for all $k$,  we can find a loop in $M$ having non-trivial intersection with the disk $\Sigma^{\alpha}_k$, a contradiction to simply connectedness. We finished the proof of the above claim.

	\textit{Claim:} $\sup_k \text{diam}_{M}(M_k) \leq C$ for some constant $C>0$.
	
	Take $y, z\in M_{k+1}$ and take the geodesic line $\overline{xy}, \overline{xz}$. Now use  $\overline{xy}\cap \p E_{k} \neq \emptyset$ and $\overline{xy}\cap \p E_{k+1} \neq \emptyset$ together with the first claim, we get that $\overline{xy} \cap \Sigma^{\alpha}_{k}\neq \emptyset$. Let $y_k\in \overline{xy}\cap \Sigma^{\alpha}_k$. Similarly, we get $\overline{xz} \cap \Sigma^{\alpha}_k \neq \emptyset$ and $z_k\in \overline{xz}\cap \Sigma^{\alpha}_k$. Now we have,
	\begin{equation*}
		|\overline{yz}|\leq |\overline{yy_k}|+|\overline{y_kz_k}|+|\overline{z_kz}| \leq 6\pi+(\pi+4)+6\pi=C.
	\end{equation*}
	Using Lemma \ref{sffBound}  and Lemma \ref{volWBG}, we get that there is a constant $C_0>0$ such that $\sup_k \Vol(M_k)\leq C_0$.
	
	We now use a suitable cut-off function to show the non-parabolic end of $M$ has linear volume growth.
	
	Let $\varphi_k$ be a cut-off function such that $\varphi_k \rvert_{B_{3k\pi}}(x)=1, \varphi_k \rvert_{(M \setminus B_{6k\pi}(x))}=0$ and with $|\nabla\varphi_k| \leq \frac{2}{3k\pi}$ everywhere and $\varphi_k \rvert_{P_m}$ is constant for $k\leq m\leq 2k$.
	
	On the parabolic ends $P_m$ we use the existence of harmonic functions converging to $1$ everywhere, with Dirichlet energy converging to $0$ (see \cite{li2004lectures}, \cite{chodosh2024complete}, \cite{wu2023free}). That is, let $u^l_m$ be harmonic functions on $P_m$, with $u^l_m \rvert_{\p' P_m}=1$, $u^l_m \rvert_{P_m} \rightarrow 1$ as $l \rightarrow \infty$, and $\int_{P_m}|\nabla u^l_m|^2 \rightarrow 0$. We may further assume that, by choosing $l$ large (depending on $k,m$), we have $\int_{P_m}|\nabla u^l_m|^2 \leq \frac{1}{k}2^{-m}$.
	
	We define $v^l_k: M \rightarrow[0,1]$ such that $v^k_l \rvert_{P_m}=u^l_m$ for $m \leq 2k$, and $v^l_k$ is constant otherwise.
	
	Let $\phi=\varphi_k v_k^l$, then recall the stability inequality of $M$,
	\begin{align*}
		&\int_{M}(\Ric_X(\nu_M,\nu_M)+|\sff_M|^2) \varphi_k^2 (v_k^l)^2+\int_{\p M} \sff_{\p X}(\nu_{M},\nu_{M})\varphi_k^2 (v_k^l)^2\\
		&\leq 2\int_{M}|\nabla \varphi_k |^2 (v_k^l)^2+2|\nabla v_k^2|^2 \varphi_k^2\\
		&\leq  2\sum_{m=k}^{2k} (\frac{2}{3k\pi})^2 \int_{M_m} (v_k^l)^2+2\sum_{m=1}^{2k}\int_{P_m} |\nabla v_k^l|^2 \\
		&\leq \sum_{m=k}^{2k} \frac{10}{k^2} \Vol(M_m)+2 \sum_{m=1}^{2k} \frac{1}{k}2^{-m}\\
		&\leq (10 C_0 + 4)\frac{1}{k} \rightarrow 0 \quad \text{as } k \rightarrow 0.
	\end{align*}
	Since $\varphi_k v_k^l \rightarrow 1$ everywhere, and all terms on the left hand side of the inequality is non-negative by assumption, we obtain the desired rigidity results.
\end{proof}

\bibliographystyle{alpha}
\bibliography{REF.bib}

\end{document}